\documentclass[a4paper,12pt]{article}

\newenvironment{proof}{\noindent {\bf Proof:}}{\hfill $\Box$}

\newtheorem{theorem}{Theorem}
\newtheorem{lemma}{Lemma}

\newtheorem{definition}{Definition}
\newtheorem{assumption}{Assumption}

\usepackage{amsmath}
\usepackage{amssymb}
\usepackage{amsfonts}
\usepackage{mathrsfs}
\usepackage{graphicx}
\usepackage{xcolor, enumerate}
\usepackage[normalem]{ulem}
\usepackage{hyperref}

\newcommand{\cH}{\mathcal H}
\newcommand{\bM}{\mathbf M}
\newcommand{\RR}{\mathbb R}
\newcommand{\NN}{\mathbb N}
\newcommand{\C}{\mathcal C}
\newcommand{\uu}{\mathbf u}
\newcommand{\vv}{\mathbf v}

\definecolor{didiergreen}{rgb}{0.1, 0.7, 0.1}

\title{\bf Infinite-dimensional moment-SOS hierarchy for nonlinear partial differential equations}

\begin{document}

\author{
Didier Henrion$^{1,2}$,
Maria Infusino$^3$,
Salma Kuhlmann$^4$,
Victor Vinnikov$^5$
}
\footnotetext[1]{CNRS; LAAS; Universit\'e de Toulouse, 7 avenue du colonel Roche, F-31400 Toulouse, France. }
\footnotetext[2]{Faculty of Electrical Engineering, Czech Technical University in Prague,
Technick\'a 2, CZ-16626 Prague, Czechia.}
\footnotetext[3]{Department of Mathematics and Computer Science, University of Cagliari, Via Ospedale 72, Palazzo delle Scienze 09124 Cagliari, Italy.}
\footnotetext[4]{Department of Mathematics and Statistics, University of Konstanz, 78457 Konstanz, Germany.}
\footnotetext[5]{Department of Mathematics, Ben-Gurion University of the Negev, P.O.B. 653, Be'er Sheva 8410501 Israel.}

\date{\today}

\maketitle

\begin{abstract}
We formulate a class of nonlinear {evolution} partial differential equations (PDEs) as linear optimization problems on moments of positive measures supported on infinite-dimensional vector spaces. Using sums of squares (SOS) representations of polynomials in these spaces, we can prove convergence of a hierarchy of finite-dimensional semidefinite relaxations solving approximately these infinite-dimensional optimization problems. As an illustration, we report on numerical experiments for solving the heat equation subject to a nonlinear perturbation.
\end{abstract}

\section{Introduction}
The \emph{moment-sum-of-squares (SOS) hierarchy} is a mathematical technology that consists of formulating a nonconvex nonlinear mathematical problem as a linear convex optimization problem on the cone of positive measures, and then solving approximately the infinite-dimensional linear problem on measures by a hierarchy of finite-dimensional convex  (and typically semidefinite) optimization problems called moment relaxations. This hierarchy builds on the duality between the cone of positive moments and the cone of positive polynomials, and its convergence relies on SOS representations of positive polynomials. See \cite{hkl20} for a recent overview and applications.

The moment-SOS hierarchy, also known as the Lasserre hierarchy, \cite{l01,hl03} was originally used in the early 2000s to solve globally finite-dimensional polynomial optimization problems \cite{l09}, and then it was extended to optimal control of nonlinear ordinary differential equations \cite{lhpt08,hp17,hkw19}. It was later on extended to optimal control of linear partial differential equations (PDEs) \cite{mp20}. More recently, it was further extended to nonlinear PDEs, for scalar hyperbolic conservation laws \cite{mwhl20} as well as in a more general semialgebraic setting \cite{khl21}.

The approach followed in \cite{mwhl20,khl21} consists of introducing linear optimization problems with measure-valued solutions, with the aim that the measures are concentrated on the solution of the PDE. Concentration is achieved for scalar hyperbolic conservation laws using additional linear entropy inequalities \cite{mwhl20}. In the more general setting of \cite{khl21}, concentration cannot be ensured, and there may be a relaxation gap between the nonlinear PDE and the linear problem on measures, unless the problem is convex \cite{hkkr23}.

Prior to these recent attempts, in \cite{mknt08} the authors discretized nonlinear PDEs in the time and space domain, producing a large-scale but sparse non-convex polynomial optimization problem that can be solved with the sparse moment-SOS hierarchy \cite{w08,mw23}. In \cite{mp20} the authors focused on the optimal control of linear {evolution} PDEs with Riesz spectral operators, i.e. with a discrete real spectrum. In this setting, the infinite-dimensional evolution equation was discretized and approximated with a finite-dimensional ordinary differential equation, and the standard finite-dimensional moment-SOS hierarchy is applied as in \cite{lhpt08,hp17,hkw19}.

In this paper we follow a radically different route to solving nonlinear PDEs with the moment-SOS hierarchy. We use measures supported on infinite-dimensional vector spaces to formulate the nonlinear differential equation as a moment problem. The vector space is chosen such that, on the one hand, we have a representation theorem (a Positivstellensatz) to deal with polynomial positivity, and on the other hand, we can solve numerically the infinite-dimensional problem with finite-dimensional sections of the semidefinite cone. As a result, approximated moments (also called pseudo-moments) of the measure solution are obtained at the price of solving a hierarchy of semidefinite optimization problems of increasing size. If the linear moment problem has a unique solution, then we can prove that our pseudo-moments converge to it.

Measures supported on infinite-dimensional spaces can be used to formulate nonlinear PDEs as linear transport problems, but so far there has been very few attempts to deal numerically with these objects, see \cite{fmw21} and references therein. In the scope of the moment-SOS hierarchy, we are still missing a systematic extension of the finite-dimensional setup to the infinite-dimensional setup. The only exception seems to be the recent work \cite{st22} which focuses exclusively on the dual SOS side and presents itself as an infinite-dimensional extension of the finite-dimensional setup of \cite{tgd18}, without attempting numerical implementation.

The purpose of this paper is to settle the ground for a genuine, systematic application of the moment-SOS hierarchy to numerically solve nonlinear PDEs as linear optimization problems on measures supported on infinite-dimensional spaces.

The outline of our paper is as follows. In Section \ref{sec:setting} we introduce our functional analytic setting and the class of nonlinear evolution PDEs we consider. Then we describe our three step infinite-dimensional moment-SOS hierarchy approach:
\begin{description}
\item Step 1 - \emph{Reformulation}: we express in Section \ref{sec:reformulation} the nonlinear PDE as a linear equation on an occupation measure supported on the solution. Under the assumption of non relaxation gap, this provides us with an equivalent linear reformulation of the nonlinear PDE. In turn, the linear measure equation is reformulated as a linear equation satisfied by the moment functions of the occupation measure;
\item Step 2 - \emph{Representation}: in Section \ref{sec:representation}, positivity and analyticity conditions are enforced on the moment functions so that they uniquely represent the occupation measure. This allows for a mathematically sound application of the moment-SOS hierarchy;
\item Step 3 - \emph{Implementation}: we use approximation properties of our solution space to formulate our moment function conditions in terms of scalar valued moments in Section \ref{sec:implementation}. 
\end{description}
This allows us to numerically approximate the solution of the nonlinear PDE as closely as desired at the price of solving a family of convex semidefinite optimization problems of increasing size, as explained in Section \ref{sec:momentsos}. 
In Section \ref{sec:heat} we illustrate in detail our approach in the case of the heat equation, a well-studied evolution PDE. We start with the linear heat equation, so that we can check consistency of the numerical solution with the known analytic solution. Then we introduce quadratic nonlinear perturbations and we illustrate that the numerics behave consistently. 

\section{Nonlinear evolution PDE}\label{sec:setting}

Let $\cH$ be a topological vector space of real valued functions in a real variable and let us denote by $\cH'$ the topological dual of $\cH$. Let $\cH_0$ be a linear subspace of $\cH$ endowed with a topology that makes the inclusion of $\cH_0$ in $\cH$ continuous and let $F: \cH_0 \rightarrow \cH_0$ be a continuous operator. 

We consider the following PDE:
\begin{equation}\label{generalPDE}
\frac{\partial u(t,x)}{\partial t} = F( u(t, x) ), \quad t \in [0,1], \quad x \in [-\pi,\pi], 
\end{equation}
with given initial condition $$u(0,.) = u_0 \in \cH_0\>$$ and periodic boundary condition $$u(t,-\pi) = u(t,\pi).$$ 
Let $\mathbb{T}:=\mathbb{R}/2\pi\mathbb{Z}$ denote the torus, i.e. the interval $[-\pi,\pi]$ with identified end points. 
Here for each $t\in [0,1],$ the function
$
u(t, \cdot): \mathbb{T} \rightarrow \mathbb{R}\>, x\mapsto u(t,x)\>
$
belongs to~$\cH_0$. 

Typical choices for $\cH$ and $\cH_0$ are the space of periodic distributions on $\mathbb{T}$ and a Sobolev space on $\mathbb{T}$, respectively. 

The operator $F$ is a polynomial function of $u$, its derivatives and integrals w.r.t. $x$. Typical examples are the viscous Burgers operator
\[
F(u(t,x))=-\frac{1}{2} \frac{\partial (u(t,x)^2)}{\partial x}+ \varepsilon \frac{\partial^2 u(t,x)}{\partial x^2}
\]
or the heat operator with a distributed nonlinearity
\[
F(u(t,x))=\frac{\partial^2 u(t,x)}{\partial x^2} + \varepsilon \int_T u(t,x)f_1(x)dx \int_T u(t,x)f_2(x)dx
\]
for given $\varepsilon>0$, $f_1 \in \cH'$, $f_2 \in \cH'$ .
More details on the choice of $F$ will follow in Section \ref{sec:reformulation}.

Denoting by $\uu$ the function
$[0,1] \rightarrow \cH_0,\> t\mapsto u(t, \cdot)$,
we view PDE \eqref{generalPDE} as an \emph{infinite-dimensional evolution equation}:
\begin{equation}\label{general-ev}
\frac{d\uu(t)}{dt} = F(\uu(t)).
\end{equation}

Throughout the paper we assume that this equation has a well-defined unique solution.
\begin{assumption}\label{existence}
{\bf (Existence and uniqueness)}
Given $u_0 \in \cH_0$, evolution equation \eqref{general-ev} has a unique solution $\uu\in \C^1([0,1] ;  \cH_0)$ with $\uu(0)=u_0$.
\end{assumption}

\section{Reformulation}\label{sec:reformulation}

\subsection{Linear measure equation}
{
Consider the time-dependent Dirac measure $\mu_t = \delta_{\uu(t)}$ supported on the solution $\uu$ of \eqref{general-ev} and let $\phi \colon [0,1] \times \cH \to {\mathbb R}$ be a test function in $\C^\infty([0,1]\times \cH; \mathbb{R})$. Since $\mu_0 = \delta_{\uu(0)}$ and $\mu_1 = \delta_{\uu(1)}$, on the one hand we have
\begin{multline} \label{lin_meas_ref-1}
\int_0^1 \frac{d}{dt}\phi(t,\uu(t)) \, dt =
\phi(1,\uu(1)) - \phi(0,\uu(0)) =
\int_{\cH} \phi(1,h) \, d\mu_1(h) - \int_{\cH} \phi(0,h) \, d\mu_0(h).
\end{multline}
On the other hand, using first the chain rule and then the fact that $\mu_t = \delta_{\uu(t)}$, we obtain
\begin{eqnarray} \label{lin_meas_ref-2}
\int_0^1 \frac{d}{dt}\phi(t,\uu(t)) \, dt &=&
\int_0^1 \left(
\frac{\partial \phi}{\partial t}(t,\uu(t))  +
\frac{\partial \phi}{\partial h}(t,\uu(t))
\frac{d}{dt} \uu(t)
\right)dt,\nonumber\\
&=& \int_0^1 \int_{\cH} \left(
\frac{\partial \phi}{\partial t}(t,h) +
\frac{\partial \phi}{\partial h}(t,h) F(h) \right) 
\, d\mu_t(h) \, dt,
\end{eqnarray}
where $\frac{\partial \phi}{\partial h}$ is the Fr\'echet derivative of  $\phi$ with respect to $h \in \cH$, which is a linear functional on $\cH$,
and $\frac{d}{dt}\uu(t) = \frac{\partial u}{\partial t}(t,x) = F(\uu(t))$ by \eqref{general-ev}.

Comparing \eqref{lin_meas_ref-1} and \eqref{lin_meas_ref-2}, we get a linear measure reformulation
\begin{multline} \label{lin_meas_ref}
\int_{\cH} \phi(1,h) \, d\mu_1(h) - \int_{\cH} \phi(0,h) \, d\mu_0(h) =
\int_0^1 \int_{\cH} \left(
\frac{\partial \phi}{\partial t}(t,h) +
\frac{\partial \phi}{\partial h}(t,h) F(h) \right) \, d\mu_t(h) \, dt,
\end{multline}
for any test function $\phi\in\C^\infty([0,1]\times \cH; \mathbb{R})$. We could therefore prove the following result.

\begin{theorem}\label{th1}
Let $\uu\in \C^1([0,1] ;  \cH_0)$
  be a solution to evolution equation \eqref{general-ev}, then the measure $\mu_t = \delta_{\uu(t)}$ on $\cH_0$ satisfies linear equation \eqref{lin_meas_ref}. Conversely, if $\mu_t = \delta_{\vv(t)}$  satisfies \eqref{lin_meas_ref} for some $\vv\in\C^1([0,1] ;  \cH_0)$, then setting $\uu :=\vv $ gives a solution $\uu$ of \eqref{general-ev} with $\uu(0)=u_0$.
\end{theorem}
}

The measure $\mu_t$ is referred to as \emph{Young measure}, or parametrized measure. This class of measures is widely studied in convex relaxations of nonlinear calculus of variations and optimal control problems, see e.g., \cite{y69}, \cite{r97} or \cite[Part III]{f99}. The measure $dt\,d \mu_t$ is called \emph{occupation measure}, because it measures the time a trajectory occupies a given set. It is classically used in Markov decision processes and dynamical systems, as well as, for the application of the moment-SOS hierarchy to nonlinear optimization and control problems, see e.g. \cite{lhpt08,hp17,hkw19,mp20,mwhl20,hkl20,khl21,hkkr23}.

\subsection{No relaxation gap}\label{sec:norelaxationgap}

Theorem 1 establishes that a measure supported on a solution of evolution equation \eqref{general-ev} solves linear equation \eqref{lin_meas_ref}. However, it may happen that a measure solving the linear equation \eqref{lin_meas_ref} is not supported on a solution of evolution equation \eqref{general-ev}. To preclude this situation, we make the following assumption.

\begin{assumption}\label{nogap}
{\bf (No relaxation gap)}
Linear equation \eqref{lin_meas_ref} with given initial data $\mu_0 = \delta_{\uu(0)}$ has a unique solution $\mu_t = \delta_{\uu(t)}$.
\end{assumption}

Under Assumptions 1 and 2, it follows from Theorem \ref{th1} that the solution of linear equation \eqref{lin_meas_ref} is concentrated on the solution of evolution equation \eqref{general-ev}. In other words, solving  \eqref{lin_meas_ref} solves \eqref{general-ev}.  If we cannot ensure Assumption~2 a priori, we can enforce the absence
of a relaxation gap by one of the following approaches:
\begin{itemize}
\item formulating the PDE as a convex minimization problem, see e.g. \cite{b20};
\item adding (linear or convex) conditions, e.g., entropy inequalities \cite{mwhl20};
\item assume that $F$ is convex, or consider its convex envelope \cite{hkkr23}.
\end{itemize}

\subsection{Linear moment equation}\label{sec:moments}

We proceed now to rewrite \eqref{lin_meas_ref} as a linear equation on the moments of the occupation measure. To this purpose let us define the concept of moment function for a measure supported on $\RR\times \cH_0$ where $\cH_0$ is a Hilbert space in $\cH$. For this, we adapt to our setting the notion of polynomials and then of moment function given in \cite{ikr14, i16, ik20}.

Let $\mathscr{P}$ be the set of all polynomials in the variables $t\in\RR$ and $h\in\cH_0$ of the form
\begin{equation}\label{poly}
P(t, h) := \sum_{\ell=0}^N\sum_{k=0}^M\langle t^{\otimes \ell}, c_\ell\rangle\langle h^{\otimes k}, f^{(k)}\rangle,
\end{equation}
where $c_\ell, f^{(0)}\in\RR$ and $f^{(k)}\in(\overline{\cH_0^{\otimes k}})'$, for $\ell=0, \ldots, N$ and $k=0,\ldots,M$ with $N, M\in\NN$. Here $\langle \cdot, \cdot\rangle$ denotes the duality between the considered vector spaces and their corresponding dual spaces. By convention, $\langle t^{\otimes 0}, c\rangle=1$ for all $c\in\RR$ and $\langle h^{\otimes 0}, f\rangle=1$ for all $f\in(\overline{\cH_0^{\otimes k}})'$.


Note that here all tensor products are intended to be symmetric and that, since
$\cH_0$ is a Hilbert space, the completed tensor product always corresponds to the Hilbert space tensor norm.

\begin{definition}
	{\bf (Moment function)}
Let $\ell, k\in\NN$.  Given a Radon measure $\mu$ defined on $\RR\times\cH$ and supported on $\RR\times\cH_0$ such that
\begin{equation}\label{fin-mom}
\int_{\RR\times\cH}|\langle t, 1\rangle|^\ell|\langle h, f\rangle|^k d\mu(t,h)<\infty, \ \forall f\in\cH',
\end{equation}
its \emph{$(\ell, k)-$th moment function} is the symmetric continuous multilinear functional $m^{\nu}_{\ell, k}$ on $\RR^{\otimes \ell}\oplus(\overline{\cH_0^{\otimes k}})'$ such that
\begin{equation}
\langle m^{\nu}_{\ell, k}, 1\oplus f^{(k)} \rangle=\int_{\RR \times \cH} \langle t, 1\rangle^\ell \langle h^{\otimes k},  f^{(k)} \rangle d\mu(t,h),\ \forall f^{(k)}\in(\overline{\cH_0^{\otimes k}})'.
 \end{equation}
By convention, $m^{\nu}_{0, 0}:=\mu(\RR\times\cH_0)$.
\end{definition}

\begin{definition}
{\bf (Determinacy)}
A Radon measure $\mu$ defined on $\RR\times\cH$ and supported on $\RR\times\cH_0$ is said to be \emph{determinate} if for any other Radon measure $\nu$ defined on $\RR\times\cH$ and supported on $\RR\times\cH_0$ such that $m^{\mu}_{\ell, k}=m^{\nu}_{\ell, k}$ for all $\ell, k\in\NN$, it holds $\nu=\mu$.
\end{definition}

In \eqref{lin_meas_ref} let us now consider the test function 
\begin{equation} \label{test_tensor}
\phi(t,h) := t^\ell \lambda_1(h) \cdots \lambda_k(h),
\end{equation}
where $\lambda_1,\ldots,\lambda_k\in\cH'$, and $\ell, k \in\NN$.  Note that since the measures in \eqref{lin_meas_ref} are assumed to fulfill \eqref{fin-mom} , the integrals appearing in \eqref{lin_meas_ref}  are all finite.


If $\cH$ is the space of periodic distributions on $\mathbb{T}$, {then}
\begin{equation} \label{pairing}
\lambda_j(h) = \langle h, \psi_j \rangle = \int_\mathbb{T} h(x) \psi_j(x) \, dx,
\end{equation}
where $\psi_j$ are $C^\infty$ periodic functions on $\mathbb{T}$.

The Frechet derivative of $\phi$ with respect to $h$ at the point $g\in\cH$ is given by
\begin{equation} \label{test_tensor_derivative}
{\frac{\partial \phi}{\partial h}(t,h)(g)} = t^\ell \ \sum_{j=1}^k \lambda_1(h) \cdots \widehat{\lambda_j(h)} \cdots \lambda_k(h) \lambda_j(g)
\end{equation}
where we used the notation
\[
\lambda_1(h) \cdots \widehat{\lambda_j(h)} \cdots \lambda_k(h) := \frac{\prod_{i=1}^k \lambda_i(h)}{\lambda_j(h)}.
\]

Substituting \eqref{test_tensor} {and \eqref{test_tensor_derivative}} into the linear equation \eqref{lin_meas_ref}, we get:
\begin{multline} \label{1st_discr-pre}
1^\ell \int_{\cH} \lambda_1(h) \cdots \lambda_k(h) \, d\mu_1(h) -
0^\ell \int_{\cH} \lambda_1(h) \cdots \lambda_k(h) \, d\mu_0(h) \\ =
\int_0^1 \int_{\cH} \ell t^{\ell-1}
\lambda_1(h) \cdots \lambda_k(h) \, d\mu_t(h) \, dt \\ +
\int_0^1 \int_{\cH} t^\ell \
\sum_{j=1}^k \lambda_1(h) \cdots \widehat{\lambda_j(h)} \cdots \lambda_k(h) \lambda_j(
F(h)) \, d\mu_t(h) \, dt.
\end{multline}
Choosing $\lambda_k$ as in \eqref{pairing}, we obtain the linear equation:
\begin{multline} \label{1st_discr}
1^\ell \int_{\cH} \langle h, \psi_1 \rangle \cdots \langle h, \psi_k \rangle \, d\mu_1(h) -
0^\ell \int_{\cH} \langle h, \psi_1 \rangle \cdots \langle h, \psi_k \rangle \, d\mu_0(h) \\ =
\int_0^1 \int_{\cH} \ell t^{\ell-1}
\langle h, \psi_1 \rangle \cdots \langle h, \psi_k \rangle \, d\mu_t(h) \, dt \\ +
\int_0^1 \int_{\cH} t^\ell \
\sum_{j=1}^k \langle h, \psi_1 \rangle \cdots \widehat{\langle h, \psi_j \rangle} \cdots \langle h, \psi_k \rangle \langle F(h), \psi_j \rangle
\, d\mu_t(h) \, dt.
\end{multline}

In order to formulate this equation as a linear equation on moment functions, we make the following standing assumption on the structure of the nonlinear operator $F$ appearing in evolution equation \eqref{general-ev}.

\begin{assumption}\label{polyop}
	{\bf (Polynomial operator)}
The operator $F:\cH_0\to\cH_0$ is of the form $$F(h)=\sum\limits_{s=0}^NF_s(h^{\otimes s}),$$ where $F_s: \overline{\cH_0^{\otimes s}}\to \cH_0$ is continuous. In other words, each $F_s$ can be seen as a homogeneous polynomial in $\mathscr{P}$ which is independent of $t$.
\end{assumption}

The second summand of linear equation \eqref{1st_discr} can then be written as:
\begin{multline}
 \int_0^1 \int_{\cH} t^\ell \
\sum_{j=1}^k\sum_{s=0}^N \langle h, \psi_1 \rangle \cdots \widehat{\langle h, \psi_j \rangle} \cdots \langle h, \psi_k \rangle \langle F_s(h), \psi_j \rangle
\, d\mu_t(h) \, dt\\
=
\int_0^1 \int_{\cH} t^\ell \
\sum_{j=1}^k\sum_{s=0}^N \langle h, \psi_1 \rangle \cdots \widehat{\langle h, \psi_j \rangle} \cdots \langle h, \psi_k \rangle \langle h^{\otimes s}, F_s^*(\psi_j) \rangle
\, d\mu_t(h) \, dt\\
=
\int_0^1 \int_{\cH} t^\ell \
\sum_{j=1}^k\sum_{s=0}^N \langle h^{\otimes j-1+s}, \psi_1 \otimes\cdots\otimes \widehat{\psi_j}\otimes\cdots \psi_k\otimes F_s^*(\psi_j) \rangle
\, d\mu_t(h) \, dt\\
=
\sum_{j=1}^k\sum_{s=0}^N \langle m_{\ell,  j-1+s}, \psi_1 \otimes\cdots\otimes \widehat{\psi_j}\otimes\cdots \psi_k\otimes F_s^*(\psi_j) \rangle
\end{multline}
where $F_s^*:\cH_0'\to \overline{\cH_0^{\otimes s}}$ denotes the adjoint of $F_s$. Hence equation \eqref{1st_discr} becomes a linear equation on moment functions:
\begin{multline} \label{1st_discr-mom}
 \langle m_{\ell, k}^{\delta_{1}\mu_1}, 1\oplus(\psi_1\otimes\cdots\otimes\psi_k) \rangle  -
 \langle m_{\ell, k}^{\delta_{0}\mu_0}, 1\oplus(\psi_1\otimes\cdots\otimes\psi_k) \rangle =\\
 \langle m_{\ell-1, k}^{dt\mu_t}, \ell \oplus(\psi_1\otimes\cdots\otimes\psi_k) \rangle +
\sum_{j=1}^k\sum_{s=0}^N \langle m^{dt\mu_t}_{\ell,  j-1+s}, \psi_1 \otimes\cdots\otimes \widehat{\psi_j}\otimes\cdots \psi_k\otimes F_s^*(\psi_j) \rangle.
\end{multline}

{We just proved the following reformulation of Theorem \ref{th1}:}

\begin{theorem}\label{th2}
Equation \eqref{lin_meas_ref} always implies equation \eqref{1st_discr-mom}. The converse holds  provided the measures $dt\,d\mu_t$, $\delta_{1}\,\mu_1$ and $\delta_0\,\mu_0$ are determinate.
\end{theorem}

\section{Representation}\label{sec:representation}

So far, we have reformulated linear measure equation \eqref{lin_meas_ref} as linear moment function equation \eqref{1st_discr-mom}. Theorem \ref{th2} states that the two equations are equivalent provided all the involved measures are determinate, i.e. they are uniquely determined by their moment functions. In this section, we consider the case when $\cH$ is the space ${\mathscr D}(\mathbb{T})$ of periodic distributions on $\mathbb{T}$ and $\cH_0$ is a Sobolev space $W_2^{-k}(\mathbb{T})$ (see below). In Theorem \ref{thm-pss} we provide conditions on a sequence of multilinear functionals to be the moment functions of some measure and to ensure the determinacy. 

In the following all the functions spaces are real vector spaces of real valued functions unless explicitly stated.
The space $\C^\infty(\mathbb{T})$ can be constructed as the projective limit of Sobolev spaces $W_2^{\alpha}(\mathbb{T})$ for all $\alpha\in\NN$ (see \cite[Thm 9.0.1]{garret-notes} and note that the torus is there defined as $\mathbb{R}/\mathbb{Z}$), where the Sobolev space $W_2^{\alpha}(\mathbb{T})$ is defined as the completion of $\C^\infty(\mathbb{T})$ w.r.t.\! the following norm (see e.g. \cite[p.156-159]{p65})$$\|\varphi\|_{W_2^{\alpha}(\mathbb{T})}:=\left(\sum_{n\in\mathbb Z}(1+n^2)^\alpha\left|\frac{1}{\sqrt{2\pi}}\int_\mathbb{T}\varphi(x)e^{-inx}dx\right|^2\right)^{\frac 12}.
$$
This construction induces a natural projective limit topology 
which makes it a separable nuclear space, as by \cite[Lemma 33]{Mel-diploma} for any $\alpha\in\NN$ there exists $\beta\in\NN$ such that the embedding $W_2^{\beta}(\mathbb{T})\hookrightarrow W_2^{\alpha}(\mathbb{T})$ is Hilbert-Schmidt. In fact, this holds for all $\beta>\alpha+\frac 12$. (This can be also derived by the more general results in \cite[pp.162-172]{p65}).
Let us denote by ${\mathscr D}(\mathbb{T})$ the space $\C^\infty(\mathbb{T})$ equipped with the projective limit topology, and let us denote its dual by ${\mathscr D}'(\mathbb{T})$. Then for any $s\in\RR$ we define $$W_2^s(\mathbb{T}):=\left\{f\in{\mathscr D}'(\mathbb{T}):
\|f\|_{W_2^s(\mathbb{T})}<\infty\right\}.$$
In the following we identify the dual of $W_2^s(\mathbb{T})$ with $W_2^{-s}(\mathbb{T})$. 

\begin{theorem}\label{thm-pss}
Let $m:=(m_{\ell, k})_{\ell, k\in\NN}$ be such that each $m_{\ell, k}$ is a multilinear functional on {$\RR^\ell\times\left(\C^\infty(\mathbb{T})\right)^k$} symmetric in the first $\ell$ variables as well as in the second $k$ variables.
If

\begin{enumerate}
\item 
{
\begin{description}
\item $\displaystyle\sum_{\ell,k,\ell',k'}\langle m_{\ell+\ell', k+k'},  (c_\ell c_{\ell'})\oplus (f^{(k)}\otimes f^{(k')}) \rangle\geq 0,$\\

\item $\displaystyle\sum_{\ell,k,\ell',k'}\langle m_{\ell+\ell'+1, k+k'}-m_{\ell+\ell'+2, k+k'},  (c_\ell c_{\ell'})\oplus (f^{(k)}\otimes f^{(k')}) \rangle\geq 0,$\\
\end{description}
}
{$\forall c_\ell, c_{\ell'}\in\RR,  f^{(k)}\in \C^\infty(\mathbb{T}^{k}), f^{(k')}\in \C^\infty(\mathbb{T}^{k'}),$}

\item $$ \sum\limits_{\ell=1}^\infty\frac{1}{\sqrt[2\ell]{\langle m_{2\ell,0}, 1^{\otimes 2\ell}\oplus 0\rangle}} =\infty$$ and there exists a countable total subset $E$ of {$\C^\infty(\mathbb{T})$} such that
$$\sum_{k=1}^\infty \frac{1}{\sqrt[2k]{\langle m_{0,2k}, 0\oplus f^{\otimes 2k} \rangle}} =\infty, \quad \forall\ f\in E,$$

\item for each $\ell, k\in\NN$, there exists $j_{\ell, k}\in\NN$ such that
$m_{\ell, k}$ is $\|\cdot\|^\ell\times\|\cdot\|_{W_2^{j_{\ell, k}}(\mathbb{T})}^{k}-$continuous,
\end{enumerate}

then there exists a unique Radon measure $\mu$ which is supported in {$[0,1]\times W_2^{-\beta}(\mathbb{T})\subseteq \RR\times \mathscr{D}'(\mathbb{T})$} for any $\beta>\max\{j_{0,2}, j_{1,1}, j_{2,0}\}+\frac 12$, and whose sequence of moment functions is $m$.
\end{theorem}

Equivalently, if we denote by $L_m$ the Riesz functional associated to the sequence $m$, i.e.,
\begin{eqnarray*}
L_m: &\mathscr{P}&\to\RR \\
 \  &a=\displaystyle\sum_{\ell=0}^N\sum_{k=0}^M\langle t^{\otimes \ell}, c_\ell\rangle\langle h^{\otimes k}, f^{(k)}\rangle& \mapsto L_m(a):=\displaystyle\sum_{\ell=0}^N\sum_{k=0}^M\langle m_{\ell, k}, c_\ell \oplus f^{(k)} \rangle
\end{eqnarray*}
then the conditions of Theorem \ref{thm-pss} can be written in the following more concise form as:
\begin{enumerate}
\item $L_m(a^2)\geq 0, L_m(b\:a^2)\geq 0, \:\forall a\in \mathscr{P},\,
b(t, h):
=\langle t^{\otimes 1}, 1\rangle(1-\langle t^{\otimes0}, 1\rangle)\langle h^{\otimes 0},1\rangle$,
\item $$\sum_{\ell=1}^\infty \frac{1}{\sqrt[2\ell]{L_m(\langle t^{\otimes 2\ell}, 1^{\otimes 2\ell}\rangle\langle h^{\otimes {0}}, 0\rangle)}} =\infty$$
and there exists a countable total subset $E$ of {$\C^\infty(\mathbb{T})$} such that
$$\sum_{k=1}^\infty \frac{1}{\sqrt[2k]{L_m(\langle t^{\otimes 0}, 1\rangle\langle h^{\otimes {2k}}, f^{\otimes {2k}}\rangle)}} =\infty, \quad \forall\ f\in E,$$
\item for each $\ell, k\in\NN$, there exist $j_{\ell, k}\in\NN$ and $Q_{\ell, k}>0$ such that
{$$\left|L_m(\langle t^{\otimes \ell}, c_\ell\rangle\langle h^{\otimes k}, f^{(k)}\rangle)\right |\leq Q_{\ell, k}|c_\ell|\| f^{(k)}\|_{W_2^{j_{\ell, k}}(T)}^{k},  \ \forall c\in\RR, f\in \C^\infty(\mathbb{T}).$$}

\end{enumerate}

\begin{proof}
In \cite[Theorem 3.3]{ikkm23} the authors establish solvability conditions for the moment problem on the symmetric tensor algebra $S(V)$ of a nuclear space $V$ and so for measures supported on the topological dual of $V$. Applying this general result for {$V=\RR\oplus \C^\infty(\mathbb{T})$ } endowed with the product topology given by the euclidean topology on $\RR$ and the projective topology on {$\C^\infty(\mathbb{T})$}  mentioned above, we obtain the desired conclusion. Note the $d-$th homogeneous component $S(V)_d$ of $S(V)$ is given by
$S(V)_d=\bigoplus\limits_{\ell+k=d} S(\RR)_\ell\otimes S(\C^\infty(\mathbb{T}))_k$. 
	\end{proof}

Note that in Theorem \ref{thm-pss}-3 the boundedness is imposed for every $\ell, k\in \NN$, while the support of the representing measure is contained in $[0,1]\times W_2^{-\beta}(\mathbb{T})$, i.e. the support is controlled by the bound for $\ell, k\in \NN$ such that $\ell+k=2$ (for further details see the proof of \cite[Theorem 3.3-(4)]{ikkm23}). 
 Also notice that, since the support of the representing measure is in $[0,1]\times W_2^{-\beta}(\mathbb{T})$, the positivity conditions in Theorem \ref{thm-pss}-1 in the $t$ variable are in accordance with the representability conditions in Hausdorff Theorem \cite{haus21}, i.e. a linear $L:\RR[X]\to \RR$ has a $[0,1]-$representing measure if and only if $L(p^2)\geq 0$ and $L(X(1-X)p)\geq 0$ for any $p\in\RR[X]$.

\section{Implementation}\label{sec:implementation}

Theorem \ref{thm-pss} and linear equation \eqref{1st_discr-mom} are formulated in terms of the moment functions. Now we proceed towards a computer implementation by replacing the moment functions by scalar valued moments. Note that in the following definition the moment functions are applied to complex valued test functions, resulting in complex valued functionals which are invariant under complex conjugation. It is straightforward that the complex version of Theorem \ref{thm-pss} holds, where in condition 1 for the Riesz functional $a^2$ is replaced by $a\bar{a}$ and in condition 2 for the Riesz' functional $f^{\otimes 2k}$ is replaced by $f^{\otimes k}\otimes \bar{f}^{\otimes k}$.

\begin{definition}\label{def:moment}
	{\bf (Moment)}
Given a Radon measure $\mu$ on $\RR\times{\mathscr D}'(\mathbb{T})$ with moment function sequence $m:=(m_{\ell, k})_{\ell, k\in\NN}$, its \emph{moment} indexed by $\ell \in {\mathbb N}, n_1, \ldots, n_k \in {\mathbb Z}$ for $k \in {\mathbb N}$ is the complex number 
\begin{equation}\label{moment-y}
y_{\ell,n_1,\ldots,n_k} :=\langle m_{\ell,k}, 1\oplus(\psi_1\otimes\cdots\otimes\psi_k)\rangle=
\int_{\mathbb R} \int_{{\mathscr D}'(\mathbb{T})}  t^\ell \langle h, \psi_1 \rangle \cdots \langle h,  \psi_k  \rangle  \, d\mu(h,t)
\end{equation}
where the test functions $\psi_j\in \C^\infty(\mathbb{T}; \mathbb{C})$ are set to the complex exponentials
\begin{equation}\label{basis}
\psi_j(x) := \frac{1}{2 \pi}e^{-i\, n_j x}.
\end{equation}
By convention, for $k=0$ we set $y_{\ell,n_1,\ldots,n_k}=y_{\ell, \emptyset}$.
\end{definition}


This particular class of test functions is dense in $\C^\infty(\mathbb{T}; \mathbb{C})$. The complex exponentials are the eigenfunctions of the Laplacian operator on $\mathbb{T}$. There may be other choices more adapted to the other operators. For example, if the linear part of the polynomial operator $F$ is a second order differential operator (with appropriate self-adjoint boundary conditions), we can replace the complex exponentials by the eigenfunctions of the corresponding Sturm-Liouville problem.

In the above definition each $$ h_{n_j}:=\langle h, \psi_j \rangle$$ is a Fourier coefficient of the periodic distribution $h$ expressed with its Fourier series
\begin{equation}\label{fourierseries}
h(x) = \sum_{n \in \mathbb Z} h_n e^{{\mathbf i}nx}.
\end{equation}

In the moment index ${\ell,n_1,\ldots,n_k}$,
integer $\ell$ is called the \emph{time degree}. Integer $k$ is called the \emph{algebraic degree}, whereas integer $\max_{j=1,\ldots,k}|n_j|$ is called the \emph{harmonic degree}.

Notice that the moment sequence $y$ satisfies the following symmetry properties:
\begin{align}
\label{moms_conj}
\overline{y_{\ell,n_1,\ldots,n_k}} &= y_{\ell,-n_1,\ldots,-n_k},\\
\label{moms_perm}
y_{\ell,n_{\sigma(1)},\ldots,n_{\sigma(k)}} &= y_{\ell,n_1,\ldots,n_k} \quad
\forall \sigma\in S_k,
\end{align}
where the bar denotes the complex conjugate and $S_k$ denotes the symmetric group, the set of all permutations of $k$ elements.

More specifically, for our evolution problem we define the \emph{initial moments}, \emph{terminal moments}, and \emph{occupation moments} respectively by
\begin{align}
\label{init_moms}
y^0_{\ell,n_1,\ldots,n_k} &:=\langle m_{\ell, k}^{\delta_0\mu_0}, {1\oplus} (\psi_1\otimes\cdots\otimes\psi_k)\rangle=
0^\ell \int_{{\mathscr D}'(\mathbb{T})} h_{n_1} \cdots h_{n_k} \, d\mu_0(h),\\
\label{term_moms}
y^1_{\ell,n_1,\ldots,n_k} &:=\langle m_{\ell, k}^{\delta_1\mu_1}, {1\oplus}(\psi_1\otimes\cdots\otimes\psi_k)\rangle=
1^\ell \int_{{\mathscr D}'(\mathbb{T})}  h_{n_1} \cdots h_{n_k} \, d\mu_1(h),\\
\label{occup_moms}
y^{[0,1]}_{\ell,n_1,\ldots,n_k} &=\langle m_{\ell, k}^{dt \mu_t}, 1\oplus(\psi_1\otimes\cdots\otimes\psi_k)\rangle=
\int_0^1 t^\ell \int_{{\mathscr D}'(\mathbb{T})} \,  h_{n_1} \cdots h_{n_k}  \, d\mu_t(h) \, dt.
\end{align}


For each $s=1, \ldots, N$, the adjoint homogeneous operator appearing in the linear moment function equation \eqref{1st_discr-mom} takes the concrete form:
\begin{equation}\label{Fs}
F_s^*(\psi_j)=\sum_{r_1,\ldots, r_s \in \mathbb Z}a_{r_1,\ldots, r_s}^{s,j}\psi_{r_1}\ldots\psi_{r_s}\end{equation}
 for some $a_{r_1,\ldots, r_s}^{s,j}\in\RR$ and $r_1,\ldots, r_s\in\mathbb{Z}$.
With these notations, the linear moment function equation \eqref{1st_discr-mom} becomes the linear moment equation:
\begin{equation} \label{linear_moms_ref}
y^1_{\ell,n_1,\ldots,n_k} - y^0_{\ell,n_1,\ldots,n_k} =
\ell y^{[0,1]}_{\ell-1,n_1,\ldots,n_k} +
\sum_{j=1}^k\sum_{s=0}^N \sum_{r_1,\ldots, r_s  \in \mathbb Z}a_{r_1,\ldots, r_s}^{s,j}y^{[0,1]}_{\ell,n_1,\ldots,n_{j-1}, r_1, \ldots, r_s, n_{j+1}, \ldots, n_k}
\end{equation}
where $\ell,k\in\NN, n_1,\ldots,n_k\in\mathbb{Z}$.

We can now reformulate the conditions of Theorem \ref{thm-pss} (which imply the existence of the representing measure) in terms of a moment sequence
$y:=(y_{\ell, n_1, \ldots, n_k})_{\ell, k\in \NN} \subset {\mathbb C}$.

We first consider the positivity conditions appearing in Theorem \ref{thm-pss}-1 and reformulate them in terms of moment matrices associated to the sequence $y=(y_{\ell, n_1, \ldots, n_k})_{\ell, k\in \NN}$.

Given $\ell \in \NN$, let us construct a Hermitian matrix whose entries are indexed by
$n_1,\ldots,n_k,$ $n'_1,\ldots,n'_{k'}\in{\mathbb Z}$ for $k,k' \in \NN$, namely
$$
\bM_\ell(y) := \left[y_{\ell,n_1,\ldots,n_k,-n'_1,\ldots,-n'_{k'}}\right]_{n_1,\ldots,n_k,n'_1,\ldots,n'_{k'} \in \mathbb Z}.
$$
Then consistenly with \cite{l01,l10} or \cite[Chapter 1]{hkl20} we define the \emph{moment matrix}	
$$
\bM(y) := \left[\bM_{\ell+\ell'}(y)\right]_{\ell,\ell' \in \NN}
$$
and the \emph{localizing matrix}
$$
\bM(t(1-t)y) = \left[\bM_{\ell+\ell'+1}(y)- \bM_{\ell+\ell'+2}(y) \right]_{\ell,\ell'\in \NN}.
$$
The positivity conditions become positive semidefiniteness conditions $\bM(y) \succeq 0$ and $\bM(t(1-t)y)\succeq 0$ which are infinite-dimensional linear matrix inequalities in the moment sequence $y$, compare with e.g. \cite[Theorem 3.2]{l10}.
 
Next we turn to the growth conditions appearing in Theorem \ref{thm-pss}-2:
$$\sum_{\ell=1}^\infty \frac{1}{\sqrt[2\ell]{y_{2\ell, \emptyset} }} = \infty$$
and taking
$E:= \left(\frac{1}{2\pi}e^{-\mathbf{i} n x}\right)_{n\in \mathbb Z}$
yields for each $n \in \mathbb Z$
$$\sum_{k=1}^\infty \frac{1}{\sqrt[2k]{y_{0, n,\cdots,n, -n,\cdots,-n} }} = \infty$$
where for each $k$, in the expression under the root, the term $n$ and $-n$ appear $k$ times each. This can be interpreted as an infinite-dimensional generalization of the multivariate Carleman condition, see e.g. \cite{Carl26} and \cite{Nuss}.
	
Finally, we deal with the growth conditions appearing in Theorem \ref{thm-pss}-(3). 
For each $\ell, k\in\NN$, there exist $j_{\ell, k}\in\NN$ such that
$$\sum_{n_1, \ldots, n_k \in \mathbb Z} (1+n_1^2)^{-j_{\ell, k}}\cdots(1+n_k^2)^{-j_{\ell, k}} |y_{\ell, n_1, \ldots, n_k}|^2<\infty.$$

Theorem \ref{thm-pss} can be reformulated in terms of the moments as follows.

\begin{theorem}\label{thm-pssy}
Let $y:=(y_{\ell,n_1,\ldots,n_k}) \subset \mathbb C$ be a sequence indexed by 
$\ell\in\NN, n_1,\ldots,n_k\in\mathbb{Z}$ for $k \in \mathbb N$ such that \eqref{moms_conj} and \eqref{moms_perm} hold. If
\begin{enumerate}
\item $\bM(y) \succeq 0$, $\bM(t(1-t)y) \succeq 0$,
\item $$\sum_{\ell=1}^\infty \frac{1}{\sqrt[2\ell]{y_{2\ell, \emptyset} }} = \infty$$
and for each $n \in \mathbb Z$
$$\sum_{k=1}^\infty \frac{1}{\sqrt[2k]{y_{0, n,\cdots,n, -n, \cdots, -n} }} = \infty$$
where for each $k$, in the expression under the root, the term $n$ and $-n$ appear $k$ times each,
\item for each $\ell, k\in\NN$, there exist $j_{\ell, k}\in\NN$ such that
$$\sum_{n_1, \ldots, n_k \in \mathbb Z} (1+n_1^2)^{-j_{\ell, k}}\cdots(1+n_k^2)^{-j_{\ell, k}} |y_{\ell, n_1, \ldots, n_k}|^2<\infty$$
\end{enumerate}
then there exists a unique Radon measure $\mu$ which is supported in {$[0,1]\times W_2^{-\beta}(\mathbb{T})\subseteq \RR\times \mathscr{D}'(\mathbb{T})$} for any $\beta>\max\{j_{0,2}, j_{1,1}, j_{2,0}\}+\frac 12$, and whose moment sequence is $y$.
\end{theorem}

\begin{proof}
For any $\ell, k\in\NN$, we use \eqref{moment-y} to define the linear functional $m_{\ell, k}$ on the linear span in ${\mathbb C}^{\otimes \ell} \otimes (\C^\infty(\mathbb{T}; \mathbb{C}))^{\otimes k}$ of all $1\oplus\psi_1\otimes\cdots\otimes \psi_k$ for $\psi_1, \ldots, \psi_k$ as in \eqref{basis} ($n_1, \ldots, n_k\in\mathbb Z$). By condition 3 and using the canonical identifications 
$$\left(W_2^{j_{\ell, k}}(\mathbb{T}; \mathbb{C}))^{\otimes k}\right)' \cong\left(W_2^{j_{\ell, k}}(\mathbb{T}; \mathbb{C}))'\right)^{\otimes k}\cong\left(W_2^{-j_{\ell, k}}(\mathbb{T}; \mathbb{C}))\right)^{\otimes k},$$
$m_{\ell, k}$ extends to a continuous linear functional on $\mathbb{C}^{\otimes \ell} \otimes W_2^{j_{\ell, k}}(\mathbb{T}; \mathbb{C}))^{\otimes k}$. By \eqref{moms_perm}, $m_{\ell, k}$ is a multilinear functional on {$\mathbb{C}^\ell\times\left(W_2^{j_{\ell, k}}(\mathbb{T}; \mathbb{C}))\right)^{k}$} symmetric in the first $\ell$ variables as well as in the second $k$ variables. By \eqref{moms_conj}, $m_{\ell, k}$ is invariant under complex conjugation, i.e. it defines a real valued multilinear functional on {$\mathbb{R}^\ell\times\left(W_2^{j_{\ell, k}}(\mathbb{T})\right)^{k}$}. Using the density of the complex exponentials and the continuity of $m_{\ell, k}$, condition 1 implies that Theorem \ref{thm-pss}-1 holds. By condition 2, Theorem \ref{thm-pss}-2 holds for
$E:= \left(\frac{1}{2\pi}e^{-\mathbf{i} n x}\right)_{n\in \mathbb Z}$. Hence, Theorem \ref{thm-pss} gives the desired conclusion.
\end{proof}

\section{Approximation with the moment-SOS hierarchy}\label{sec:momentsos}

Let us summarize our developments so far. We have shown that under Assumptions \ref{existence}, \ref{nogap} and \ref{polyop}, given the moments $y^0$ of an initial occupation measure concentrated on an initial data, if there are sequences $y^{[0,1]}$ and $y^1$ solving linear equation \eqref{linear_moms_ref} and satisfying the conditions of Theorem \ref{thm-pssy}, then they are the sequences of moments of an occupation measure  and a terminal measure concentrated on the solution of PDE \eqref{generalPDE}.

In order to implement the moment-SOS hierarchy on a computer, we proceed as in the finite-dimensional case, see e.g. \cite[Chapter 2]{hkl20}. We just truncate the sequences up to a given time degree, algebraic degree and harmonic degree. The resulting conditions on the truncated sequences is a finite-dimensional \emph{semidefinite optimization} problem that can be solved on a computer with e.g. interior-point algorithms. Since the moment problem is truncated, there are less variables and constraints, and we speak of a \emph{moment relaxation}. 

The solution to the moment relaxation are truncated sequences of \emph{pseudo-moments}, in the sense that they are not necessarily moments of measures: the truncated semidefinite constraints are necessary, but not sufficient conditions for the sequences to have representing measures. However, when the truncation degrees go to infinity, the conditions become necessary and sufficient, and we speak of \emph{convergence} of the moment-SOS hierarchy. Quantifying the speed of convergence, as a function of the time degree, algebraic degree and harmonic degree, is an interesting open research question.

In this paper we are only interested in solving nonlinear PDEs, there is no functional to be minimized. In practice, we can however control the growth of the pseudo-moment sequences in each moment relaxation by minimizing the sum of the traces of the moment matrices of the occupation measure and terminal measure. 

The moment relaxation, a finite-dimensional semidefinite optimization problem, has a dual semidefinite optimization problem which can be interpreted as a polynomial SOS problem, as in the finite-dimensional case, see e.g. \cite[Chapter 2]{l10} or \cite[Chapter 2]{hkl20}. This SOS interpretation lies however out of the scope of this paper.

It was already observed that the moment sequences satisfy symmetry properties \eqref{moms_conj} and \eqref{moms_perm}. Symmetries can be exploited to reduce to size of the semidefinite relaxations, and hence to improve scalability of the approach when the truncation degrees increase. Sparsity of the moment and localizing matrices, with their specific Hankel and Toeplitz structure, as well sparsity of the linear moment equation, can also be exploited by semidefinite programming solvers.

Finally, once pseudo-moments are available, it is desirable to recover approximately the graph of the solution of the PDE. For that purpose, we can use the Christoffel-Darboux kernel \cite{lpp22}: this SOS polynomial constructed from the pseudo-moments has small values on the graph of the function to be recovered. This graph recovery technique was studied in detail in \cite{m21}. In the context of PDE solving, it would be interesting to study its convergence properties.

\section{Case study: the  heat equation}\label{sec:heat}

\subsection{Linear heat equation}

Let us first consider the linear heat equation \eqref{generalPDE} with
$F(u(t,x)) = \frac{\partial^2 u(t,x)}{\partial x^2}.$

\subsubsection{Linear moment equation}

Using the duality pairing \eqref{pairing} and noticing that
\begin{equation}\label{eq:2nd-der}
\langle h^{\prime\prime}, \psi \rangle = \langle h, \psi^{\prime\prime} \rangle,
\end{equation}
linear equation \eqref{1st_discr} reads
\begin{multline*}
1^\ell \int_{\cH} \langle h, \psi_1 \rangle \cdots \langle h, \psi_k \rangle \, d\mu_1(h) -
0^\ell \int_{\cH} \langle h, \psi_1 \rangle \cdots \langle h, \psi_k \rangle \, d\mu_0(h) \\ =
\int_0^1 \int_{\cH} \ell t^{\ell-1}
\langle h, \psi_1 \rangle \cdots \langle h, \psi_k \rangle \, d\mu_t(h) \, dt \\ +
\int_0^1 \int_{\cH} t^\ell \
\sum_{j=1}^k \langle h, \psi_1 \rangle \cdots \widehat{\langle h, \psi_j \rangle} \cdots \langle h, \psi_k \rangle \langle h, \psi_j^{\prime\prime} \rangle
\, d\mu_t(h) \, dt.
\end{multline*}
In the duality pairing, we choose complex exponentials
$\psi_j(x) = \frac{1}{2\pi}e^{-\mathbf{i} n_j x}$, where $n_1,\ldots,n_k \in {\mathbb Z}$
so that corresponding linear functionals are the Fourier coefficients of $h \in \cH$ denoted
$h_n := \langle h, \frac{1}{2\pi}e^{-\mathbf{i} n x} \rangle$.
Since
$\left(e^{-\mathbf{i} n x}\right)^{\prime\prime} = - n^2 e^{-\mathbf{i} n x}$,
the linear equation becomes
\begin{multline}\label{2nd_discr} 
1^\ell \int_{\cH} h_{n_1} \cdots h_{n_k} \, d\mu_1(h) -
0^\ell \int_{\cH} h_{n_1} \cdots h_{n_k} \, d\mu_0(h) \\ =
\int_0^1 \int_{\cH} \ell t^{\ell-1}
h_{n_1} \cdots h_{n_k} \, d\mu_t(h) \, dt -
\sum_{j=1}^k n_j^2
\int_0^1 \int_{\cH} t^\ell \
h_{n_1} \cdots h_{n_k}
\, d\mu_t(h) \, dt.
\end{multline}
Following Definition \ref{def:moment} and notation \eqref{init_moms}--\eqref{occup_moms}, the initial, terminal and occupation moments are given by
\begin{align*}
y^0_{\ell,n_1,\ldots,n_k} &=
0^\ell \int_{\cH} h_{n_1} \cdots h_{n_k} \, d\mu_0(h),\\
y^1_{\ell,n_1,\ldots,n_k} &=
1^\ell \int_{\cH} h_{n_1} \cdots h_{n_k} \, d\mu_1(h),\\
y^{[0,1]}_{\ell,n_1,\ldots,n_k} &=
\int_0^1 \int_{\cH} t^\ell \, h_{n_1} \cdots h_{n_k} \, d\mu_t(h) \, dt,
\end{align*}
and we can rewrite the linear moment equation \eqref{linear_moms_ref} as
\begin{equation}\label{heatmom}
y^1_{\ell,n_1,\ldots,n_k} - y^0_{\ell,n_1,\ldots,n_k} =
\ell y^{[0,1]}_{\ell-1,n_1,\ldots,n_k} - \sum_{j=1}^k n_j^2 \ y^{[0,1]}_{\ell,n_1,\ldots,n_k}.
\end{equation}
Here $\ell,k = 0,1,2,\ldots,\,n_1,\ldots,n_k=0,\pm 1,\pm 2,\ldots$.

Notice that in the case $\ell=0$, the first term on the right hand side does not appear since the test function does not depend on $t$.
Notice also that $y^0_{\ell,n_1,\ldots,n_k}=0$ for $\ell>0$
and that $y^1_{\ell,n_1,\ldots,n_k}$ is independent of $\ell$.

\subsubsection{Analytic solution}

The solution $u$ can be expressed as a Fourier series:
\[
u(t,x) = \sum_{n \in \mathbb Z} u_n(t) e^{\mathbf{i}nx}.
\]
From the linear heat equation $$\frac{\partial u(t,x)}{\partial t}=\frac{\partial^2 u(t,x)}{\partial x^2}$$
it follows that
$$\frac{d}{dt} u_n(t)=-n^2 u_n(t)$$
for any $n$, which yields
$$u_n(t)=u_n(0)e^{-n^2 t}.$$
Therefore the solution reads
$$u(t,x) = \sum_{n \in \mathbb Z} u_n(0) e^{-n^2t}e^{\mathbf{i}nx}.$$

\begin{lemma}\label{analyticmoments}
Given $\ell, k \in \mathbb{N}$ and $n_1,\ldots,n_k \in \mathbb{Z}$, let $N:=n^2_1+\cdots+n^2_k$. The initial, terminal and occupation moments solving the linear equation \eqref{heatmom} can be expressed in closed-form as follows:
\begin{align*}
y^0_{\ell,n_1,\ldots,n_k} &= 0^{\ell} u_{n_1}(0) \cdots u_{n_k}(0),\\
y^1_{\ell,n_1,\ldots,n_k} &= 1^{\ell} u_{n_1}(0) \cdots u_{n_k}(0) e^{-N},\\
y^{[0,1]}_{\ell,n_1,\ldots,n_k}&=
\left\{
\begin{array}{ll}
\frac{u_0(0)^k}{\ell+1}& \text{if $(n_1, \ldots, n_k)=({0,\ldots,0})$}\\
u_{n_1}(0) \cdots u_{n_k}(0)\left(\frac{\ell!}{N^{\ell+1}}-e^{-N}\sum\limits_{j=1}^{\ell+1} \frac{\ell!}{N^j(\ell-j+1)!}\right)& \text{if $(n_1, \ldots, n_k)\neq (0,\ldots,0).$}
\end{array}
\right.
\end{align*}
\end{lemma}

\begin{proof}
The initial and terminal moments follow readily from the Fourier series expansion.
The occupation moments are given by
$$y^{[0,1]}_{\ell,n_1,\ldots,n_k} = u_{n_1}(0) \cdots u_{n_k}(0) \int_0^1 t^{\ell} e^{-Nt}dt.$$
First consider the case $(n_1, \ldots, n_k)=({0,\ldots,0})$. Then, using the definiton of $y_{\ell, 0,\ldots,0}$ and observing that in this case $N=0$, we easily get:
\[
y_{\ell, 0,\ldots,0}=\underbrace{u_0(0)\cdots u_0(0)}_{k\:\text{times}}\int_0^1 t^{\ell} dt=\frac{u_0(0)^k}{\ell+1}.
\]
For the case $(n_1, \ldots, n_k)\neq({0,\ldots,0})$, we set for convenience $I_{\ell}:=\int_0^1 t^{\ell} e^{-Nt}dt$. Then the conclusion will follow at once by
\begin{equation}\label{I_l}
    I_{\ell}=\frac{\ell!}{N^{\ell+1}}-e^{-N}\sum\limits_{j=1}^{\ell+1} \frac{\ell!}{N^j(\ell-j+1)!},
\end{equation}
which by induction holds for all $\ell\in\mathbb{N}$.
Indeed, a direct computation of the integral $I_{\ell}$ provides the base case, that is, \[
I_0=\int_0^1 e^{-Nt}dt=-\frac{e^{-Nt}}{N}\bigg|_{t=0}^{t=1}=\frac{1}{N}-\frac{e^{-N}}{N},\]
i.e. \eqref{I_l} holds for $\ell=0$.
Suppose now that \eqref{I_l} holds for $\ell=l$ with $l\geq 0$ and let us prove it for $\ell=l+1$. Integrating by parts and using the definition of $I_l$ we obtain that
\begin{align*}
    I_{l+1}&=\int_0^1 t^{l+1} e^{-Nt}dt\\
           &=-\frac{t^{l+1}e^{-Nt}}{N}\bigg|_{t=0}^{t=1}+\int_0^1(l+1)t^l\frac{e^{-Nt}}{N}dt\\
           &=\frac{1}{N}\left((l+1)I_l-e^{-N}\right).
\end{align*}
This by inductive assumption becomes
\begin{align*}
    I_{l+1}&=\frac{1}{N}\left(\frac{l!(l+1)}{N^{l+1}}-e^{-N}\sum\limits_{j=1}^{l+1} \frac{l!(l+1)}{N^j(l-j+1)!}-e^{-N}\right)\\
    &=\frac{(l+1)!}{N^{l+2}}-e^{-N}\sum\limits_{j=1}^{l+1} \frac{(l+1)!}{N^{j+1}(l-j+1)!}-\frac{e^{-N}}{N}\\
     &=\frac{(l+1)!}{N^{l+2}}-e^{-N}\sum\limits_{j=0}^{l+1} \frac{(l+1)!}{N^{j+1}(l-j+1)!}\\
     &=\frac{(l+1)!}{N^{l+2}}-e^{-N}\sum\limits_{r=1}^{l+2} \frac{(l+1)!}{N^{r}(l+1-r+1)!},
\end{align*}
where in the last step we used the change of variable $r=j+1$. Hence, by induction, \eqref{I_l} holds for all $\ell\in\mathbb{N}$ which yields the conclusion.
\end{proof}

\subsection{Nonlinear heat equation}

\subsubsection{Distributed nonlinearity}

Consider a nonlinear heat PDE \eqref{generalPDE} with a distributed quadratic nonlinearity, e.g.
\begin{equation} \label{eq:heat-nonlin-dist}
F(u(t,x)) = \frac{\partial^2 u(t,x)}{\partial x^2} + \epsilon \int_{-1}^1 u(t,x)f_1(x)dx \int_{-1}^1 u(t,x)f_2(x)dx
\end{equation}
where $f_1, f_2 \in {\mathcal C}^\infty(T)$ are given functions, and $\epsilon \in \RR$ is a given small number.

Let us choose e.g. the real valued functions $f_1=\psi_{m_1}+\psi_{-m_1}$, $f_2=\psi_{m_2}+\psi_{-m_2}$ for $m_1,m_2 \in {\mathbb N}$ and $\psi_{m_j}(x) = \frac{1}{2\pi} e^{-{\mathbf i}{m_j} x}$ as in Section \ref{sec:implementation}.

	In linear moment equation \eqref{1st_discr} we substitute
	$F(h) = h^{\prime\prime} + \epsilon (h_{m_1}+h_{-m_1})(h_{m_2}+h_{-m_2}) \cdot 1$
	(we could in principle replace the constant function $1$ by any real valued trigonometric polynomial). Using \eqref{eq:2nd-der} we have
	$$
	\langle F(h), \psi_j \rangle = \langle h, \psi_j^{\prime\prime} \rangle + \epsilon (h_{m_1}+h_{-m_1})(h_{m_2}+h_{-m_2}) \langle 1, \psi_j \rangle.
	$$
	Using further $\phi_n^{\prime\prime} = -n^2 \phi_n$ and $\langle 1, \phi_n \rangle = \delta_{n0}$ (the orthogonality of the complex exponentials) we obtain instead of \eqref{2nd_discr} the following:
	\begin{multline} \label{2nd_discr-heat_nonlin_dist}
	1^\ell \int_{\cH} h_{n_1} \cdots h_{n_k} \, d\mu_1(h) -
	0^\ell \int_{\cH} h_{n_1} \cdots h_{n_k} \, d\mu_0(h) \\ =
	\int_0^1 \int_{\cH} \ell t^{\ell-1}
	h_{n_1} \cdots h_{n_k} \, d\mu_t(h) \, dt -
	\sum_{j=1}^k n_j^2
	\int_0^1 \int_{\cH} t^\ell \
	h_{n_1} \cdots h_{n_k}
	\, d\mu_t(h) \, dt \\ - \epsilon
	\int_0^1 \int_{\cH} t^\ell \ \sum_{j=1}^k
	\delta_{n_j0} h_{n_1} \cdots \widehat{h_{n_j}} \cdots h_{n_k} (h_{m_1}+h_{-m_1})(h_{m_2}+h_{-m_2})
	\, d\mu_t(h) \, dt.
	\end{multline}
	In terms of moments, we obtain finally instead of \eqref{heatmom} the following:
	\begin{multline} \label{linear_moms_ref-heat_nonlin_dist}
	y^1_{\ell,n_1,\ldots,n_k} - y^0_{\ell,n_1,\ldots,n_k} =
	\ell y^{[0,1]}_{\ell-1,n_1,\ldots,n_k} - \sum_{j=1}^k n_j^2 \ y^{[0,1]}_{\ell,n_1,\ldots,n_k} \\
	- \epsilon \sum_{j=1}^k \delta_{n_j0}
	\left(y^{[0,1]}_{\ell,n_1,\ldots,\widehat{n_j},\dots,n_k,m_1,m_2} +
	y^{[0,1]}_{\ell,n_1,\ldots,\widehat{n_j},\dots,n_k,m_1,-m_2}\right. \\ \left.+
	y^{[0,1]}_{\ell,n_1,\ldots,\widehat{n_j},\dots,n_k,-m_1,m_2} +
	y^{[0,1]}_{\ell,n_1,\ldots,\widehat{n_j},\dots,n_k,-m_1,-m_2}\right)
	\end{multline}
	for $\ell,k = 0,1,2,\ldots,\,n_1,\ldots,n_k=0,\pm 1,\pm 2,\ldots$
	(and again, for $\ell=0$ the first term on the right hand side does not appear).

\subsubsection{Local nonlinearity}

Let us consider now a nonlinear heat PDE \eqref{generalPDE} with a distributed quadratic nonlinearity, e.g.
\begin{equation} \label{eq:heat-nonlin-local}
F(u(t,x))  = \frac{\partial^2 u(t,x)}{\partial x^2} + \epsilon u(t,x)^2
\end{equation}
where $\epsilon \in \RR$ is a given small number.

In linear moment equation \eqref{1st_discr} we substitute $F(h) = h^{\prime\prime} + \epsilon h^2$.
Using \eqref{eq:2nd-der} we have
$$
\langle F(h), \psi_j \rangle = \langle h, \psi_j^{\prime\prime} \rangle + \epsilon \langle h^2, \psi_j \rangle.
$$
Using the Fourier series expansion \eqref{fourierseries}, by convolution we get
\begin{equation}
h^2 = \sum_{n, m \in \mathbb Z} h_m h_{n-m} e^{\mathbf{i} n x}.
\end{equation}
Using this together with
$\phi_n^{\prime\prime} = -n^2 \phi_n$
we obtain instead of \eqref{2nd_discr} the following:
\begin{multline} \label{2nd_discr-heat_nonlin_local}
1^\ell \int_{\cH} h_{n_1} \cdots h_{n_k} \, d\mu_1(h) -
0^\ell \int_{\cH} h_{n_1} \cdots h_{n_k} \, d\mu_0(h) \\ =
\int_0^1 \int_{\cH} \ell t^{\ell-1}
h_{n_1} \cdots h_{n_k} \, d\mu_t(h) \, dt -
\sum_{j=1}^k n_j^2
\int_0^1 \int_{\cH} t^\ell \
h_{n_1} \cdots h_{n_k}
\, d\mu_t(h) \, dt \\ - \epsilon
\int_0^1 \int_{\cH} t^\ell \ \sum_{j=1}^k
h_{n_1} \cdots \widehat{h_{n_j}} \cdots h_{n_k}
\left(\sum_{m \in \mathbb Z} h_m h_{n_j-m}\right)
\, d\mu_t(h) \, dt.
\end{multline}
In terms of moments, we obtain finally instead of \eqref{heatmom} the following:
\begin{multline} \label{linear_moms_ref-heat_nonlin_local}
y^1_{\ell,n_1,\ldots,n_k} - y^0_{\ell,n_1,\ldots,n_k} =
\ell y^{[0,1]}_{\ell-1,n_1,\ldots,n_k} - \sum_{j=1}^k n_j^2 \ y^{[0,1]}_{\ell,n_1,\ldots,n_k}
- \epsilon \sum_{j=1}^k \sum_{m \in \mathbb Z}
y^{[0,1]}_{\ell,n_1,\ldots,\widehat{n_j},\dots,n_k,m,n_j-m}
\end{multline}
for $\ell,k = 0,1,2,\ldots,\,n_1,\ldots,n_k=0,\pm 1,\pm 2,\ldots$
(and again, for $\ell=0$ the first term on the right hand side does not appear). 

\subsection{Numerical experiments}

We developed Matlab codes to generate the semidefinite moment relaxations in YALMIP dual format and call the external semidefinite solver Mosek for solving the nonlinear heat PDE. The main files {\tt heatmom$\_$qd.m} for quadratic distributed nonlinearity and {\tt heatmom$\_$ql.m} for quadratic local nonlinearity, as well as their dependencies, can be downloaded at \begin{center}\href{https://homepages.laas.fr/henrion/software/heatmom/}{\tt homepages.laas.fr/henrion/software/heatmom/}\end{center}

The moment relaxations are generated by various values of the time degree, algebraic degree and harmonic degree, as defined in Section \ref{sec:implementation}.
Decision variables are the occupation moment vector $y^{[0,1]}$ and the terminal moment vector $y^1$. The initial moment vector $y^0$ is given, and it is computed from an initial condition \[
u(0,x)= \sum_{k\in \mathbb Z} u_k(0) e^{{\mathbf i} k x}\] with finitely many nonzero Fourier coefficients. For our experiments we took $$u_{(-1,0,1)}(0)=(1,1,1).$$

In Table \ref{tab:sizes} we report sizes of the moment vector and moment matrix for various degrees. General purpose semidefinite solvers on a standard computer can deal with semidefinite matrices of size of a few hundreds and a few tens of thousands of variables. For example, on our laptop, for degrees (4,2,4) and (4,4,2) the semidefinite relaxation is solved in less than 0.5 sec. For degrees (4,4,4), (6,4,4) resp. (6,4,6) it takes approx. 13, 24 resp. 544 secs.

For the linear heat PDE ($\epsilon=0$) we can compare the pseudo-moments obtained by optimization and the analytic moments calculated in Lemma \ref{analyticmoments}.
We report in Figure \ref{fig:mommatch} the relative accuracy between the pseudo-moments and the analytic occupation moments. When all the degrees are equal to 4, we can see that approx. 95\% of the pseudo-moments are approximating the exact moments with relative accuracy $10^{-5}$.

For the heat PDE with distributed nonlinearity, we do not have the analytic moments, but in Figure \ref{fig:mommatch_qd} we report on the discrepancy between the computed pseudo-moments and the analytic moments for $\epsilon=0$, for distinct values of $\epsilon$. Similarly, for the hear PDE with local nonlinearity, we report in Figure \ref{fig:mommatch_ql} on the discrepancy between the computed pseudo-moments and the analytic moments for $\epsilon=0$, for distinct values of $\epsilon$.

\section*{Acknowledgements}
This research was supported through the program "Oberwolfach Research Fellows" by the Mathematisches Forschungsinstitut Oberwolfach in 2023. The last three authors wish to thank the Banff International Research Station for the opportunity to work on this project during the workshop 23w5033 and the research in teams 23rit010. M. Infusino is indebted to the Baden--W\"urttemberg Stiftung for the financial support to this work by the Eliteprogramme for Postdocs and is a member of GNAMPA group of INdAM. S. Kuhlmann was partially supported by the Ausschuss f\"ur Forschungsfragen~(AFF) from the University of Konstanz. V. Vinnikov would like to thank Milken Families Foundation Chair in Mathematics for a partial support of his research.\\

\section{Tables and Figures}

\begin{table}[h]\begin{center}
\begin{tabular}{c|cc}
(time, algebraic, harmonic) & moment & moment \\
degrees & vector size & matrix size \\ \hline
(2, 2, 2) & 63 & 12 \\
(4, 2, 2) & 105 & 18 \\
(6, 2, 2) & 147 & 24 \\
(6, 2, 4) & 385 & 40 \\
(2, 4, 2) & 378 & 42 \\
(4, 4, 2) & 630 & 63 \\
(6, 4, 2) & 882 & 84 \\
(4, 4, 4) & 3575 & 165 \\
(6, 4, 4) & 5005 & 220 \\
(6, 4, 6) & 16660 & 420 \\
(6, 6, 4) & 35035 & 880  \\
(6, 6, 6) & 189924 & 2240
\end{tabular}
\caption{\label{tab:sizes}Moment vector and matrix sizes for various degrees.}
\end{center}\end{table}

\begin{figure}[h]\begin{center}
\includegraphics[width=0.60\textwidth]{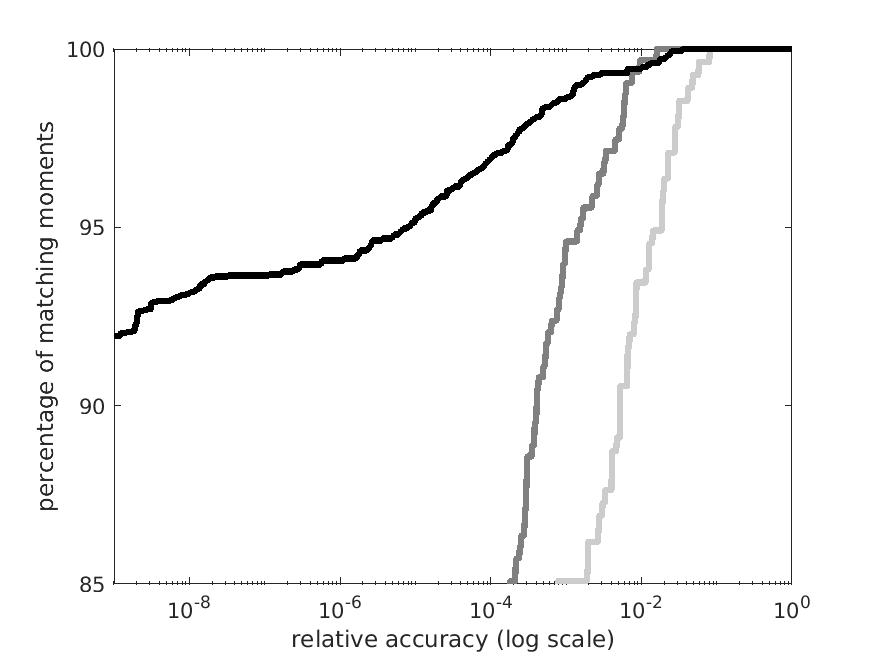}
\caption{\label{fig:mommatch}Percentage of matching occupation moments versus relative accuracy for distinct time, algebraic and harmonic degrees: (4,2,4) light gray, (4,4,2) dark gray, (4,4,4) black.}
\end{center}\end{figure}

\begin{figure}[h]\begin{center}
		\includegraphics[width=0.60\textwidth]{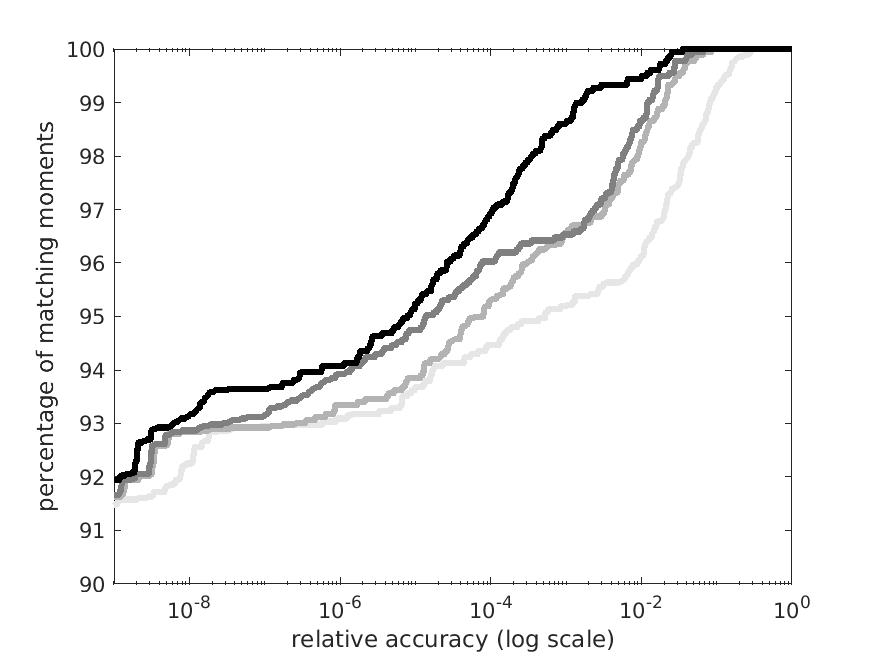}
		\caption{\label{fig:mommatch_qd}Distributed nonlinearity: percentage of matching occupation moments versus relative accuracy for distinct  values of $\epsilon$: $0$ black, $10^{-6}$ dark gray, $10^{-3}$ gray, $1$ light gray.}
\end{center}\end{figure}

\begin{figure}[t]\begin{center}
		\includegraphics[width=0.60\textwidth]{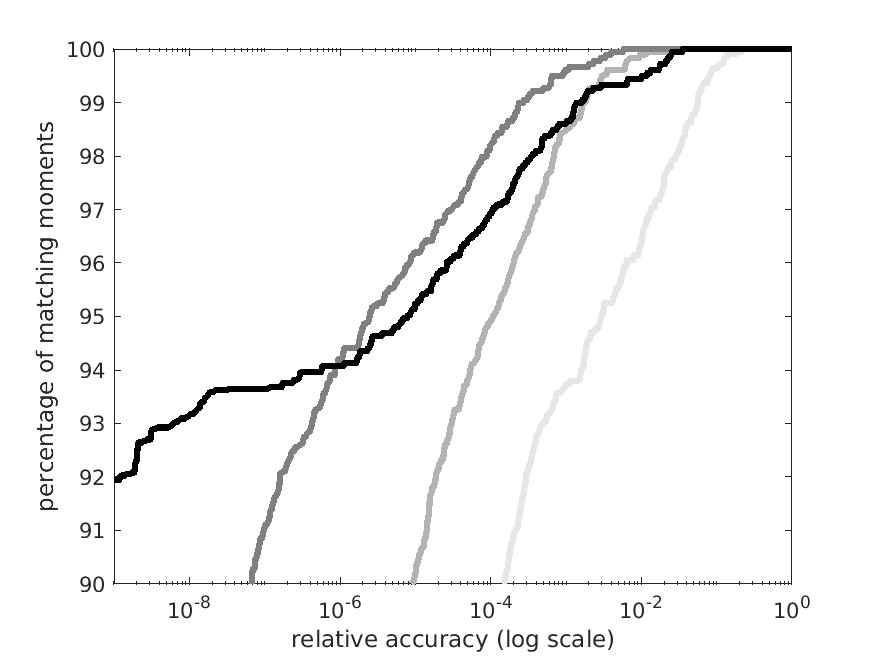}
		\caption{\label{fig:mommatch_ql}Local nonlinearity: percentage of matching occupation moments versus relative accuracy for distinct  values of $\epsilon$: $0$ black, $10^{-6}$ dark gray, $10^{-3}$ gray, $1$ light gray.}
\end{center}\end{figure}

\end{document}